\documentclass[11pt]{article}

\usepackage{amssymb,amsmath,amsfonts,epsf,amsmath}
\usepackage{enumerate}
\usepackage[ruled,vlined,linesnumbered]{algorithm2e}

\usepackage[english]{babel}
\usepackage{diagbox,comment,xparse}

\usepackage{tikz,multicol}

\usetikzlibrary{decorations.pathreplacing,automata,calc,positioning}
\usetikzlibrary{arrows}

\newtheorem{thm}{Theorem}[section]

\newtheorem{conj}[thm]{Conjecture}
\newtheorem{cor}[thm]{Corollary}
\newtheorem{lema}[thm]{Lemma}
\newtheorem{claim}{Claim}
\newtheorem{remark}[thm]{Remark}

\newtheorem{prob}{Problem}

\newenvironment{claimproof}[1]{{\it\noindent{Proof.}}\space#1}{\footnotesize \hfill \ensuremath{(\square)} \medskip}

\newcommand{\efface}[1]{}
\newcommand{\proof}{\noindent{\bf Proof.\ }}
\newcommand{\qed}{\hfill $\square$ \medskip}

\newcommand{\spq}{{\sigma_{(p,q)}}}
\newcommand{\senaq}{{\sigma_{(1,q)}}}

\def\cp{\,\square\,}

\begin{document}

\title{Spreading in graphs}

\author{Bo\v stjan Bre\v sar$^{a,b}$, Tanja Dravec$^{a,b}$, Aysel Erey$^{c}$, and
Jaka Hed\v zet$^{b,a}$}

\date{}

\maketitle

\begin{center}
$^a$ Faculty of Natural Sciences and Mathematics, University of Maribor, Slovenia\\
\medskip

$^b$ Institute of Mathematics, Physics and Mechanics, Ljubljana, Slovenia\\
\medskip

$^c$ Department of Mathematics and Statistics, Utah State University, USA \\
\end{center}

\begin{abstract}
Several concepts that model processes of spreading (of information, disease, objects, etc.) in graphs or networks have been studied. In many contexts, we assume that some vertices of a graph $G$ are contaminated initially,  before the process starts. By the $q$-forcing rule, a contaminated vertex having at most $q$ uncontaminated neighbors enforces all the neighbors to become contaminated, while by the $p$-percolation rule, an uncontaminated vertex becomes contaminated if at least $p$ of its neighbors are contaminated. If given a set $S$ of initially contaminated vertices all vertices eventually become contaminated when continuously applying the $q$-forcing rule (respectively the $p$-percolation rule), $S$ is called a $q$-forcing set (respectively, a $p$-percolating set) in $G$. In this paper, we consider sets $S$ that are at the same time $q$-forcing sets and $p$-percolating sets, and call them $(p,q)$-spreading sets. Given positive integers $p$ and $q$, the minimum cardinality of a $(p,q)$-spreading set in $G$ is a $(p,q)$-spreading number, $\sigma_{(p,q)}(G)$, of $G$. While $q$-forcing sets have been studied in a dozen of papers, the decision version of the corresponding graph invariant has not been considered earlier, and we fill the gap by proving its NP-completeness. This, in turn, enables us to prove the NP-completeness of the decision version of the $(p,q)$-spreading number in graphs for an arbitrary choice of $p$ and $q$.  Again, for every $p\in \mathbb{N}$ and $q\in\mathbb{N}\cup\{\infty\}$, we find a linear-time algorithm for determining the $(p,q)$-spreading number of a tree, where in the case $p\ge 2$ we apply Riedl's algorithm from [Largest and smallest minimal percolating sets in trees,
Electron.\ J.\ Combin.\ 19 (2012) Paper 64] on $p$-percolation in trees. In addition, we present a lower and an upper bound on the $(p,q)$-spreading number of a tree and characterize extremal families of trees. In the case of square grids, we combine some results of Bollob\'{a}s from [The Art of Mathematics: Coffee Time in Memphis. Cambridge Univ.\ Press, New York, 2006], and the AIM Minimum Rank-Special Graphs Work Group from [Zero forcing sets and the minimum rank of graphs, Linear algebra Appl.\ 428 2008 1628--1648], and new results on $(2,1)$-spreading and $(4,q)$-spreading to obtain $\sigma_{(p,q)}(P_m\Box P_n)$ for all $(p,q)\in (\mathbb{N}\setminus\{3\})\times (\mathbb{N}\cup\{\infty\})$ and all $m,n\in\mathbb{N}$. 
\end{abstract}


\medskip
\noindent{\bf AMS Subj. Class.:}  05C35, 05C75, 05C85
\section{Introduction}

Zero forcing was independently introduced by mathematical physicists in the study of quantum systems~\cite{BG-2007} and by linear algebraists in connection with the maximum nullity of real symmetric matrices~\cite{AIM}. Since then, zero forcing received a lot of attention by graph theorists; see~\cite{bbgk, gprs-2016,kks-2019,l-2019} for some recent papers and the monograph~\cite{HLS} that surveys various studies of zero forcing. The notion of zero forcing appears also in the analysis of the diverse processes, in which the need of uniform framework led to the introduction of the generalization of zero forcing named $k$-forcing~\cite{acdp}.

Suppose that vertices of a graph $G$ are assigned one of the two possible colors,  white and blue. The color change rule in $k$-forcing is defined as follows: {\it if a blue vertex $u$ has at most $k$ white neighbors, then change color of all white neighbors of $u$ to blue}. In  particular, when $k=1$, this is the color change rule for zero forcing. 
If for a set of vertices $S \subseteq V(G)$, which are initially colored blue, continuous application of the color change rule enforces all vertices of $G$ to become blue, then $S$ is a {\em $k$-forcing set} of $G$ (resp., a {\em zero forcing set} of $G$ when $k=1$). The {\it $k$-forcing number} of $G$, $F_k(G)$, is the minimum cardinality of a $k$-forcing set of $G$, while the {\em zero forcing number}, when $k=1$, is denoted by $Z(G)$, and so $Z(G)=F_1(G)$.

The concept of bootstrap percolation has a similar definition to that of zero forcing, yet a different motivation. It was introduced as a simplified model of a magnetic system in 1979~\cite{CLR-1979}, and was later studied on random graphs and also on deterministic graphs. Let us mention that one of the first graphs in which bootstrap percolation was studied were grids~\cite{BP-1998,Bol-2006}. 

Again let all vertices of a graph $G$ be colored either white or blue. (In some applicable contexts, blue vertices are called ``infected'' and white vertices ``uninfected''.)  A set $S \subseteq V(G)$ of blue vertices is an {\em $r$-neighbor bootstrap percolating set}, or simply, an {\em $r$-percolating set} of $G$ if by repeatedly applying the rule: {\it if a white vertex $u$ has at least $r$ blue (infected) neighbors, then change color of $u$ to blue}, all vertices of $G$ eventually become blue. The {\em $r$-neighbor bootstrap percolation number} of $G$, $m(G,r)$, is the minimum cardinality of an $r$-neighbor bootstrap percolating set of $G$. 

Motivated by the similarity of the definitions of the above two concepts, we introduce a common generalization of $k$-forcing and bootstrap percolation. 
Let $p \in \mathbb{N}$ and $q \in \mathbb{N} \cup \lbrace \infty \rbrace$, and let vertices of a graph $G$ be colored either white or blue. If a white vertex $w$ has at least $p$ blue neighbors, and one of the blue neighbors of $w$ has at most $q$ white neighbors, then by the {\em spreading color change rule} the color of $w$ is changed to blue. A set $S$ is a $(p,q)$-{\em spreading set} for $G$ if initially exactly the vertices of $S$ are colored blue and by repeatedly applying the spreading color change rule all the vertices of $G$ are eventually turned to blue. The $(p,q)$-{\em spreading number}, $\spq(G)$, of a graph $G$ is the minimum cardinality of a $(p,q)$-spreading set.

The aim of this paper is to initiate the study of spreading in graphs. The paper is organized as follows. In the next section, we specify connections with some specializations of $(p,q)$-spreading, such as bootstrap percolation, $P_3$-convexity and $q$-forcing. We also give some basic preliminary results. In Section~\ref{sec:NP}, we concentrate on complexity issues regarding the computation of the $(p,q)$-spreading number. We prove that the decision version of determining $\spq(G)$, where $G$ is an arbitrary graph, is an NP-complete problem. We do this by translating the decision version of determining the $q$-forcing number of $G$, for which we also prove that it is NP-complete. The latter result was, to the best of our knowledge, not known before. In Section~\ref{sec:trees}, we consider spreading in trees. First, we present several algorithms to compute $\spq(T)$, where $T$ is a tree. We prove that $F_q(T)$ equals the size of a smallest partition of $T$ into induced subtrees of maximum degree at most $q$, which enables us to present a linear-time algorithm to determine $\sigma_{(1,q)}(T)$ for a tree $T$ and $q\in \mathbb{N}$. We also note that Riedl's algorithm~\cite{algoritem}, which computes the $p$-neighbor bootstrap percolation of a tree in linear time, where $p\ge 2$, works for any $q\in \mathbb{N}$. Second, we also find an upper and a lower bound for the $(p,q)$-spreading number of a tree and characterize the trees that attain these bounds. Section~\ref{sec:grids} is devoted to spreading in square grids. By combining some known results (say, about the 2-neighbor bootstrap percolation in grids due to Bollob\'{a}s~\cite{Bol-2006}) with some new results, we obtain the $(p,q)$-spreading number of all grids for all $p$ and $q$ with the sole exception when $p=3$. Finally, we state some concluding remarks and propose several open problems in Section~\ref{sec:conclude}.

In the rest of this section, we present the notation used throughout the paper. All graphs considered in this paper are finite, simple, and undirected. Let $[n]=\{1,2,\ldots ,n\}$, where $n \in \mathbb{N}$. For a graph $G=(V,E)$ and $S \subseteq V(G)$ we write $G[S]$ for the subgraph of $G$ induced by $S$. The {\em neighborhood} of $v \in V(G)$, $N_G(v)$, is the set of all vertices in $G$ adjacent to $v$.

Similarly, for $S \subseteq V(G)$ and $v\in V(G)$, we let $N_S(v)= \{u \in S; vu \in E(G)\}$. The {\em degree} of a vertex $v$ is deg$_G(v)=|N_G(v)|$, and $\Delta(G)=\max{\{\textrm{deg}_G(v); v \in V(G)\}}, \delta(G)=\min{\{\textrm{deg}_G(v); v \in V(G)\}}$. For $u,v \in V(G)$, $d_G(u,v)$ denotes the length of a shortest $u,v$-path in $G$. For $v \in V(G)$ and a subgraph $H \subseteq G$ of a graph $G$, $d_G(u,H)=\min{\{d_G(u,v); v\in V(H)\}}.$
A subset of vertices $S$ is called a {\it vertex cover} of $G$ if every edge of $G$ has an end-vertex in $S$. 

The  {\em Cartesian product} $G\cp H$ of graphs $G$ and $H$ is defined as the graph with $V(G\cp H)=V(G)\times V(H)$, and $(g,h)(g',h')\in E(G\cp H)$ if either ($g=g'$ and $hh\in E(H)$) or ($gg'\in E(G)$ and $h=h'$).

%%%%%%%%%%%%%%%%%%%%%%%%%%%%%
\section{Preliminary results}
\label{sec:prelim}
%%%%%%%%%%%%%%%%%%%%%%%%%

\begin{table}%[!h]
\begin{center}
\begin{tabular}{ |c|c|c| } 
 \hline
 $p$ & $q$ & $(p,q)$-spreading\\
 \hline
 \hline
 $1$ & $1$ & zero-forcing\\
% \hline
 $1$ & $q$ & $q$-forcing\\
% \hline
 $p$ & $\infty$ & $p$-bootstrap percolation\\
 $2$ & $\infty$ & $P_3$-convexity\\
 \hline
 \end{tabular}
 \caption{Specializations of $(p,q)$-spreading.}
 \label{tab:con}
 \end{center}
 \end{table}
 Note that $(p,q)$-spreading is a common generalization of bootstrap percolation and zero forcing. The $(1,k)$-spreading number is exactly the $k$-forcing number and the $(r,\infty)$-spreading number is exactly the $r$-neighbor bootstrap percolation number. For $r=2$ this is also identical to the concept that arises from $P_3$-convexity known as the $P_3$-hull number  (see~\cite{bv-2019+,CPR} for more on this concept), so the $(2,\infty)$-spreading number is at the same time the 2-bootstrap percolation number and the $P_3$-hull number. Table~\ref{tab:con} presents these connections. 

By the definition of the $(p,q)$-spreading number we immediately get the following observations. 

\begin{remark}\label{r:basic}
 Let $G=(V,E)$ be a graph with maximum degree $\Delta$ and let $S$ be a $(p,q)$-spreading set of $G$. The following statements hold;
\begin{enumerate}[(1)]
\item if $\deg(v) < p$ then $v \in S$;
\item if $\Delta < p$ then $\spq(G) = |V(G)|$;
\item if $\Delta \leq q$ then $\spq(G) = \sigma_{(p,\infty)}(G)$;
\item $\spq(G) \geq \min{\{p, |V(G)|\}}$;
\item $S$ is a $(p,q+1)$-spreading set, and $\spq(G) \geq \sigma_{(p,q+1)}(G)$;
\item $S$ is a $(p-1,q)$-spreading set, and $\spq(G) \geq \sigma_{(p-1,q)}(G)$.
\end{enumerate}
\end{remark}

Next, we determine the $(p,q)$-spreading numbers in some basic graph classes for any $p \in \mathbb{N}$ and any $q \in {\mathbb{N}} \cup \{\infty\}$. Since $\Delta(P_n)=\Delta(C_n)=2$, Remark~\ref{r:basic}(2) implies that $\spq(P_n)=\spq(C_n)=n$ for $p > 2$. 
For $p=1$, $\sigma_{(1,q)}(P_n)=F_q(P_n)=1$ and for $q \geq 2$, $\sigma_{(2,1)}(P_n) \geq \sigma_{(2,q)}(P_n) = \sigma_{(2,\infty)}(P_n)=m(P_n,2) = \lceil \frac{n+1}{2} \rceil$ by Remark~\ref{r:basic}(3) and (5). For $q \geq 2$ we again easily deduce that $\lceil \frac{n+1}{2} \rceil =\sigma_{(2,1)}(C_n) \geq \sigma_{(2,q)}(C_n)=m(C_n,2)= \lceil \frac{n}{2} \rceil$. In addition,  

\begin{displaymath}
\sigma_{(1,q)}(C_n)=F_q(C_n)=
    \begin{cases}
        2 & \text{if } q =1\\
        1 & \text{if } q > 1.
    \end{cases}
\end{displaymath}

The spreading number can be easily computed also for complete graphs. It holds

\begin{displaymath}
\spq(K_n)=
    \begin{cases}
        p & \text{if } p+q \geq n\\
        n-q & \text{if } p+q < n.
    \end{cases}
\end{displaymath}

Finally, the $(p,q)$-spreading number of a complete bipartite graph $K_{r,s}$ with $r \geq s$, depends on the relation between $p,q,r$ and $s$. For example, if $s < p \leq r$, then we get

\begin{displaymath}
\spq(K_{r,s})=
    \begin{cases}
        r & \text{if } q \geq s\\
        r+s-q & \text{if } q < s.
    \end{cases}
\end{displaymath}

We end the section with a trivial result that will be needed later.

\begin{lema} \label{lema: delta > q}
Every $(p,q)$-spreading set $S$ contains a vertex $v$ such that $$|N_S(v)|\geq \deg_G(v)-q.$$
\end{lema}
\begin{proof}
If every vertex in $S$ has more than $q$ neighbors outside $S$, then $S$ cannot be a $(p,q)$-spreading set, since the color change rule cannot begin. \qed
\end{proof}

%%%%%%%%%%%%%%%%%%%%%%%%%
\section{NP-completeness}
\label{sec:NP}
%%%%%%%%%%%%%%%%%%%%%%%%%

Given a graph $G$ and positive integers $p$ and $q$, the optimization version of the $(p,q)$-spreading problem is to find a $(p,q)$-spreading set of minimum size in $G$.  The decision version of the problem can be formulated as follows:

\begin{center}
\vskip -.5cm
\fbox{
	\parbox{\textwidth}{
		\textsc{$(p,q)$-Spreading} \\
		\textit{Input}: A graph $G = (V,E)$ and positive integers $p,q$ and $k$
		\\
		\textit{Question}: Is there a $(p,q)$-spreading set in $G$ of size at most $k$?}}
\end{center}

A special case of $(p,q)$-spreading problem is obtained when we set $p=1$, in which case $(p,q)$-spreading, which is then $(1,q)$-spreading, coincides with $q$-forcing. For some earlier studies of $q$-forcing see~\cite{acdp,fhky,ms-2021}. The corresponding graph invariant, the {\em $q$-forcing number} of a graph $G$, is denoted by $F_q(G)$ and equals $\senaq(G)$.
 We will deal with this version of the problem separately, and formulate the corresponding decision problem as follows:

\begin{center}
\vskip -.5cm
\fbox{
	\parbox{\textwidth}{
		\textsc{$q$-Forcing} \\
		\textit{Input}: A graph $G = (V,E)$ and positive integers $q$ and $k$
		\\
		\textit{Question}: Is there a $q$-forcing set in $G$ of size at most $k$?}}
\end{center}

It is clear that both problems are in NP, and it suffices to argue this for the \textsc{$(p,q)$-Spreading} problem. Let $G$ be a graph, and $p,q$ be positive integers. If a set $S$ of size $k$ is presented to us, then we can clearly check in polynomial time whether $S$ is a $(p,q)$-spreading set in $G$.

First, we prove NP-completeness of the \textsc{$q$-Forcing} problem. We will use a translation of the well-known \textsc{zero forcing} problem, which was proven to be NP-complete in~\cite{aa}. Note that the \textsc{zero forcing} problem is simply the \textsc{$q$-forcing} problem for $q=1$. The standard notation for the zero forcing number of a graph $G$ is $Z(G)$. 

We need the following definition. Let $G=(V,E)$ be a graph and $S$ a $q$-forcing set of $G$. Then by definition, before the $q$-forcing process starts vertices of $S$ are blue, vertices of $V(G)-S$ are white and the $q$-forcing procedure recolors all vertices of $V(G)-S$ to blue. In particular, in the $i$th step of this procedure, by the $q$-forcing rule, there exists a blue vertex $a_i \in V(G)$ whose white neighbors are vertices $x_1^i,\ldots , x_{q_i}^i$, where $ 1 \leq q_i \leq q$, and $a_i$ forces vertices  $x_1^i,\ldots , x_{q_i}^i$ to become blue. Let us denote the corresponding $q$-{\it forcing sequence} $(a_1,\ldots , a_k)$ by $Z$ and denote by $f_Z$ the function which maps each $a_i \in Z$ to the subset $\{ x_1^i,\ldots , x_{q_i}^i\}$ of $V(G)$.

Let $G=(V,E)$ be a graph with $V(G)=\{v_1,\ldots , v_n\}$. We define the graph $G^*$ constructed from $G$ as follows. For any vertex $v_i \in V(G)$ add vertices $a_1^i,\ldots, a_{q-1}^i$, $b_1^i,\ldots ,b_q^i, c_1^i\ldots ,c_q^i$ and edges $v_ia_j^i$ for any $j \in [q-1]$ and all possible edges between vertices in $A^i \cup B^i$ and all possible edges between vertices in $B^i \cup C^i$, where  $A^i=\{ a_1^i,\ldots, a_{q-1}^i\}, B^i=\{ b_1^i,\ldots ,b_q^i\}, C^i=\{ c_1^i\ldots ,c_q^i\}$. Hence the subgraph of $G^*$ induced by $A^i \cup B^i$ is isomorphic to the complete graph $K_{2q-1}$ and the subgraph of $G^*$ induced by $B^i \cup C^i$ is isomorphic to $K_{2q}$. Denote by $H_i$ the subgraph of $G^*$ induced by $A^i \cup B^i \cup C^i$.

\begin{lema}\label{Gstar}
If $G$ is a graph and $q\ge 2$ an integer, then $F_q(G^*)=Z(G)$.
\end{lema}
\proof
First, we will show that if $S$ is a zero forcing set of $G$, then $S$ is also a $q$-forcing set of $G^*$. If a blue vertex in $V(G)$ has at most one white neighbor in $G$, then it has at most $q$ white neighbors in $G^*$. Therefore, if a white vertex in $G$ becomes blue according to zero-forcing rule in $G$, then it becomes blue according to $q$-forcing rule in $G^*$ as well. It follows that $S$ forces all vertices in $V(G)$ to become blue according to $q$-forcing rule in $G^*$. After all vertices in $V(G)$ become blue, each vertex $v_i$ has at exactly $q-1$ white neighbors and so, each $v_i$ forces all vertices in $A^i$ to become blue according to $q$-forcing rule in $G^*$. Now, each vertex in $A_i$ has at most $q$ white neighbors. Therefore, any vertex in $A_i$ forces all vertices in $B_i$ to become blue according to $q$-forcing rule in $G^*$. Similarly, after all vertices in $B_i$ become blue, each vertex in $B_i$ has at most $q$ white neighbors. Therefore, any vertex in $B_i$ forces all vertices in $C_i$ to become blue. This implies that $S$ is a $q$-forcing set of $G^*$ and thus $F_q(G^*)\le Z(G)$.
\iffalse \textcolor{gray}{
 Let $S$ be a minimum zero forcing set of $G$. Then, $S$ is also a $q$-forcing set of $G^*$, since essentially the same procedure of color changes applies in $G^*$ as in $G$. Let $Z=(x_1,x_2,\ldots , x_k)$ be a zero forcing sequence of $G$. Note that at step $\ell$ of the zero forcing procedure in $G$ $x_{\ell} \in V(G)$ is blue and has exactly one white neighbor in $G$. Since $x_{\ell} \in V(G)$, there exists $i \in [n]$ such that $x_{\ell}=v_i$. Thus $v_i$ is blue and it has exactly one white neighbor $v_j$ in $G$. Hence in $G^*$ vertex $v_i$ has exactly $q$ white neighbors and thus they can become blue by the $q$-forcing procedure. Then the blue vertex $a_1^i \in A^i$ has exactly $q$ white neighbors (vertices of $B^i$), hence whole $B^i$ can be colored blue. Now any vertex from $B^i$ has exactly $q$ white neighbors (vertices of $C^i$) and thus $C^i$ can be colored blue. Finally, if $v_i \notin Z$, then when all vertices of $G$ are already blue, $v_i$ forces $A^i$ to become blue, a vertex from $A^i$ then forces $B^i$ to become blue and a vertex from $B^i$ forces $C^i$ to recolor to blue. This implies that $S$ is a $q$-forcing set of $G^*$ and thus $F_q(G^*)\le Z(G)$.}
\fi

Now, let $S^*$ be a minimum $q$-forcing set of $G^*$ such that $\sum_{s \in S^*}d_{G^*}(s,G)$ is the smallest possible. Denote $S^*_i=S^* \cap V(H_i)$ and let $Z^*=(x_1,\ldots , x_k)$ be a $q$-forcing sequence of $G^*$.
\begin{claim}\label{c:1}
For any $i \in [n]$: $|S^*_i| \leq q-1$ and $S^* \cap (B^i \cup C^i) = \emptyset $.
\end{claim}
\begin{claimproof}
For the purpose of contradiction assume that $|S^*_i| \geq q$. Then $S'=(S^*-S^*_i) \cup A^i \cup \{v_i\}$ is a $q$-forcing set of $G^*$ with $\sum_{s \in S^*}d_{G^*}(s,G) > \sum_{s \in S'}d_{G^*}(s,G)$, a contradiction.

Since $|S^*_i| \leq q-1$ the set $S'$, obtained from $S^*$ by replacing each vertex in $B^i \cup C^i$ by one vertex in $A^i$, is also a $q$-forcing set of $G^*$. If $(B^i\cup C^i)\cap S_i^*\neq \emptyset$, then $S'$ has a smaller sum of distances to $G$ than $S^*$ which contradicts with the definition of $S^*$. Thus, $(B^i\cup C^i)\cap S_i^*=\emptyset$.

\end{claimproof}

By Claim~\ref{c:1}, $S^* \subseteq V(G)\cup A^1 \cup \ldots \cup A^n$.

\begin{claim}\label{c:2}
If $j$ is the smallest index such that $x_j \in V(H_i)$ for some $i \in [n]$, then $x_j \in A^i$ and $v_i \notin f_{Z^*} (x_j)$.
\end{claim}
\begin{claimproof}
By Claim~\ref{c:1}, $S^*_i \cap (B^i \cup C^i)= \emptyset$, hence vertices in $B^i$ can become blue just if all vertices from $A^i$ are already blue, which implies that a vertex from $A^i$ will force vertices in $B^i$ and then vertex in $B^i$ will force vertices in $C^i$ to become blue. Hence $x_j \in A^i$, i.e.\ $x_j=a_m^i$ for some $m\in [q-1].$

Suppose that at this step of $q$-forcing procedure $v_i$ is one of at most $q$ white neighbors of $x_j=a_m^i$. Hence at least $q$ vertices of $A^i \cup B^i$ are blue at this step. Since by Claim~\ref{c:1}, $|S^*_i| \leq q-1$, we get a contradiction with the fact that $a_m^i$ is the first vertex of $Z^*$ that is also in $H^i$. 
\end{claimproof}

Claim~\ref{c:2} implies that $v_i$ is already blue when the $q$-forcing rule is performed by $x_j$.
By using $Z^*$ we will construct a zero forcing sequence $Z$ in $G$. In particular, we will show that $Z$ can be obtained from $Z^*$ by removing from $Z^*$ only the vertices from $G^*$ that are not in $G$.

Suppose that we are in step $j$ of the $q$-forcing procedure in $G^*$, where blue vertex $x_j$ in $Z^*$ will force its white neighbors  to become blue ($x_j$ has at most $q$ white neighbors). We will prove that if $x_j \in V(G)$, then it has at most one white neighbor in $V(G)$, which will yield that $Z$ is a zero-forcing sequence. 

Thus let $x_j \in V(G)$. Hence there exists $i \in [n]$ such that $x_j=v_i$. Let $A$ be the set of white neighbors of $v_i$ in $H_i$ after the first $j-1$ steps of the $q$-forcing procedure. Suppose to the contrary that $A \neq A^i$, i.e.\ at least one vertex of $A^i$ is blue. Vertices of $H_i$ that are blue after $j-1$ steps of the $q$-forcing procedure are either from $S^*$ or they were colored blue in one of these steps. If the latter holds, then by Claim~\ref{c:2}, there exists a vertex $a_t^i$ from $A^i$ which is in $Z^*$ before $x_j$.  By Claims~\ref{c:1} and \ref{c:2} we infer $S^*_i=A^i$. Hence all vertices of $A^i$ are blue before step $j$ of the $q$-forcing procedure. Since by Claim~\ref{c:2}  $v_i \notin f_{Z^*}(a_t^i)$, the set $S'$ obtained from $S^*$ by replacing a vertex from $A^i$ by one of $v_i$'s white neighbors in $G$ (which exists because $v_i=x_j \in Z^*$ and has no white neighbors in $H_i$), is also a $q$-forcing set of $G^*$. To see that $S'$ is a $q$-forcing set note that all the $q$-forcing rules in the first $j-1$ steps are preformed in the same way as with $S^*$, except possibly within $H_i$. Anyway, after $v_i$ forces its neighbors to become white, we can force the whole $H_i$ to become blue after at most two additional steps. The remaining steps, after step $j$, follow exactly the same procedure as with $S^*$.
Note that $S'$ has smaller sum of distances to $G$ than $S^*$, a contradiction.  Thus vertices of $H_i$ that are blue after the first $j-1$ steps of the $q$-forcing procedure are from $S^*$. In the same way as above, the set $S'$ obtained from $S^*$ by removing from $S^*$ a blue vertex from $A^i$ and adding to $S'$ one of $v_i$'s white neighbors in $G$ (if it exists), is also a $q$-forcing set of $G^*$. Note that it is possible that $v_i$ has no white neighbors in $G$ in which case $S'$ is smaller than $S^*$, a contradiction. Otherwise $S'$ has smaller sum of distances to $G$ than $S^*$, again a contradiction. Hence $A=A^i$, i.e.\ all vertices of $A^i$ are white before step $j$ of the $q$-forcing procedure. Hence $x_j=v_i$ has at most one white neighbor in $G$.

Hence the sequence $Z$ obtained from $Z^*$  by removing from $Z^*$ only the vertices from $G^*$ that are not in $G$, is a zero forcing sequence.
Thus $S^* \cap V(G)$ is a zero forcing set which implies that $Z(G) \leq F_q(G^*)$.
\qed

\begin{thm}
The \textsc{$q$-Forcing} problem  is NP-complete.
\end{thm} 

In other words, the \textsc{$(1,q)$-Spreading} problem is NP-complete. We will now translate this problem to the general \textsc{$(p,q)$-Spreading} problem, and show its NP-completess.
 
\begin{lema}\label{Gprime}
Let $G$ be a graph and let $p,q\in \mathbb{N}$ with $p\ge 2$. 
If $G'$ is obtained from $G$ by adding vertices $u_1,\ldots, u_{p-1}$, attaching $p$ leaves to each vertex $u_i$, and connecting each $u_i$ to every vertex of $G$, then $\spq(G')=F_q(G)+p(p-1)$.  
\end{lema}
\proof Let $S$ be a minimum $q$-forcing set of $G$. Set $U=\{u_1,\ldots,u_{p-1}\}$, and let $L$ be the set of all leaves in $G'$. (Note that $|L|=p(p-1)$.) Now, let $S'=S\cup L$ be the set of vertices that are initially set blue within the $(p,q)$-spreading process in $G'$. Note that the leaves in $L$ make their neighbors in $U$ blue, because each $u_i$ has $p$ blue leaf-neighbors. Hence, each vertex of $V(G)$ already has at least $p-1$ blue neighbors, and so the condition of $(p,q)$-spreading that each white vertex must have at least $p$ neighbors is fulfilled as soon as a white neighbor in $V(G)$ has a blue neighbor also in $V(G)$. Thus, the process of changing colors with respect to the $(p,q)$-spreading rules takes the same steps as the color change process with respect to $q$-forcing in $G$. Hence, $S'$ is a $(p,q)$-spreading set in $G'$, and we readily infer that $\spq(G')\le F_q(G)+p(p-1)$.

For the reversed inequality, let $S'$ be a minimum $(p,q)$-spreading set of $G'$. Note that necessarily $L\subset S'$, for otherwise none of the leaves would ever become blue, because they each have at most one neighbor and $p>1$. In addition, $U\cap S'=\emptyset$, because each $u_i$ becomes blue due to its $p$ leaf neighbors. Let $S=S'\cap V(G)$. By the same reasoning as in the previous paragraph, we infer that $S$ is a $q$-forcing set of $G$. Indeed, the condition for a white vertex having at least $p$ neighbors if fulfilled as soon as it has a blue neighbor in $G$. Therefore only the condition of a blue neighbor having at most $q$ white neighbors has to be considered, which is exactly the $q$-forcing rule. Hence, $S$ is a $q$-forcing set of $G$, and so $F_q(G)\le\spq(G')-p(p-1)$.
\qed

Since the construction of $G'$ from $G$ can be done in polynomial time, we infer, by using also Lemma~\ref{Gprime}, that a polynomial solution of the \textsc{$(p,q)$-Spreading} problem in $G'$ yields a polynomial solution of the  \textsc{$q$-Forcing} problem in $G$. We derive the following result.

\begin{thm}
The \textsc{$(p,q)$-Spreading} problem is NP-complete.
\end{thm}

%%%%%%%%%%%%%%%%%%%%%%
\section{Trees}
\label{sec:trees}
%%%%%%%%%%%%%%%%%%%%%%

In this section, we first consider algorithmic issues concerning $(p,q)$-spreading in trees. In particular, for all feasible values of $p$ and $q$ we present efficient algorithms to obtain a minimum $(p,q)$-spreading set of an arbitrary tree. Second, we consider upper and lower bounds for $\spq(T)$, where $T$ is a tree, and characterize families of trees that achieve these bounds.

\subsection{Algorithms for $(p,q)$-spreading}

The zero forcing problem on trees has been studied in several papers. It is well known that the zero forcing number $Z(T)$ of a tree $T$ is equal to the size of the smallest partition of the vertex set of the tree into induced paths, which is denoted by $P(T)$;  see~\cite{AIM,FH-2007}. In this section, we generalize this result to the context of $q$-forcing, for which we need the following definition. Let $P_q(T)$ denote the size of the smallest partition of $V(T)$ into subsets inducing trees of maximum degree at most $q+1$. Recall that $F_q(T)=\sigma_{(1,q)}(T)$, and by definition, $P(T)=P_1(T)$ and $Z(T)=F_1(T)$.

\begin{thm}\label{thm:q-forcing drevesa}
Let $T$ be a tree and $q \in \mathbb{N}$. Then $F_q(T)=P_q(T)$. 
\end{thm}

\begin{proof}
First, let $S=\{v_1,\ldots ,v_k\}$ be an $F_q$-set of $T$. Then a blue vertex can apply the color change rule only if it is adjacent to at most $q$ white neighbors. For each $i \in [k]$ let $T_i=\{v_i\}$. Then whenever a vertex $x$ from $T_i$ has at most $q$ white neighbors, by the color change rule these white neighbors become blue and we add them to the set $T_i$ (note that $x$ is not unique, but we may choose it arbitrarily).  Notice that since $T$ is a tree, the graphs induced by vertices in $T_i$ are also trees. Finally when every vertex is colored blue, the obtained sets $T_1,\ldots ,T_k$ form a partition of $T$ into induced trees of degree at most $q+1$, which infers $F_q(T) \geq P_q(T)$.

Now let $P_q(T)=k$ and $\{T_1,T_2,\ldots ,T_k\}$ be a smallest partition into induced trees of degree at most $q+1$. Let $\mathcal{T}$ be a graph defined as follows: $V(\mathcal{T}) = \{T_1,T_2,\ldots ,T_k\}$ and $T_iT_j \in E(\mathcal{T})$ if and only if there exist vertices $u \in T_i$, $v \in T_j$, such that $uv \in E(T)$. Then $\mathcal{T}$ is also a tree. We will use induction on $k$ to show that there exists an $F_q$-set $S$, such that $|S|=k$, therefore proving that $F_q(T) \leq |S| = P_q(T)$. First, let $k=1$, thus $\mathcal{T}$ consists only of vertex  $T_1$. Let $S=\{x\}$ where $x$ is a leaf. Then immediately the parent of $x$ is colored blue. Notice that since $\Delta(T) \leq q+1$, whenever a white vertex changes color because of a blue neighbor, it can then have at most $q$ white neighbors, therefore the process continues. Since $T$ is connected, this implies that $\{x\}$ is an $F_q$-set. 

For the induction step, let $k \geq 2$ and suppose that $S'$ is an $F_q$-set of cardinality $m$ for every tree $T'$ with $|P_q(T')| = m < k$. Let $T_k$ be a leaf in $\mathcal{T}$ and let $T_kT_1 \in E(\mathcal{T})$. Furthermore, let $u_k$ be the vertex in $T_k$, which is adjacent to a vertex $u_1 \in V(T_1)$. Let $x_k$ be a leaf in $T_k$. A similar color change process as in the base case takes place in $T_k$ - newly colored vertices have at most $q$ white neighbors. The only exception is vertex $u_k$. When $u_k$ becomes blue, it has at most $q$ white neighbors in $T_k$ and also one neighbor in $T_1$. Therefore it can have $q+1$ white neighbors in which case the process cannot continue. However, note that $u_k$ is already blue. Now let $T'=T-V(T_k)$. Then $|P_q(T')|=k-1$ and by the induction hypothesis, there exists an $F_q$-set $S'$ of $T'$, such that $|S'| = k-1$. We claim that $S=S' \cup \{x_k\}$ is an $F_q$-set of $T$. Indeed, we can start the color change process in $x_k$ and continue in $T_k$ until $u_k$ is colored by blue. Since $u_k$ is blue and as $u_ku_1$ is the only edge between $T'$ and $T_k$, we can apply the same color changing process as in $T'$ and color all vertices of the subtree $T'$ in $T$ to blue. Now, if at this point not all vertices of $T_k$ are blue, we can simply continue by applying the color change rule on $T_k$, because $u_k$ and its unique neighbor in $T'$ are already blue. In this way, also all vertices of $T_k$, and thus all vertices of $T$ become blue. With this we proved that $S$ is an $F_q$-set of $T$, and so $F_q(T) \leq P_q(T)$. 
\qed
\end{proof}

Next we will present an algorithm that computes $P_q(T)$ in a polynomial time. Start with some notations and observations.  Let $T$ be a tree. If $\Delta(T) \leq q+1$, then $P_q(T)=1$. Thus we may assume that $\Delta(T) > q+1$. We root $T$ in an arbitrary vertex $x_0$ which is the only vertex at depth 0.  Denote by $X_1$ all neighbors of $x_0$, i.e.\ all vertices at depth 1 of the rooted tree $T$. Furthermore, for any $i > 1$ denote by $X_i$ the vertices at depth $i$ of the rooted tree $T$. Denote by $d$ the depth of $T$, i.e.\ $V(T)=X_0 \cup X_1 \cup \ldots \cup X_d$. Moreover, for $xy \in E(T)$, $x \in X_i, y \in X_{i-1}$, denote by $T_x$ the connected component of $T-y$ containing $x$.

Let $i$ be the largest index such that there exists a vertex $x_i \in X_i$ with $\deg(x_i) > q+1$. If $i=0$, then $P_q(T)=\deg(x_0)-q$. Otherwise $x_i$ has exactly one neighbor, say $x_{i-1}$, in $X_{i-1}$ and $k > q$ neighbors $x_{i+1}^1,\ldots , x_{i+1}^k$ in $X_{i+1}$. In any smallest partition $\mathcal{P}$ of $T$ into induced trees of of maximum degree $q+1$, neighbors of $x_i$ are in $k-q+1$ different trees of $\mathcal{P}$ (i.e.\ $q+1$ neighbors of $x_i$ are in the same tree of $\mathcal{P}$). By the choice of $i$, $\Delta(T_{x_{i+1}^j}) \leq q+1$ for any $j \in \{1,\ldots, k\}$. Thus there are at least $k-q-1$ indices $j$ such that $T_{x_{i+1}^j} \in \mathcal{P}$.

\begin{algorithm}[hbt!]
{\caption{Smallest partition of $T$ into induced trees of max-degree $q+1$.}
\label{al:partition of a tree}
\KwIn{A rooted tree $T$ with vertices in $X_i$ at depth $i$.}
\KwOut{Smallest partition of $T$ into induced trees of maximum degree $q+1$.}

\BlankLine
{
	$\mathcal{P}=\emptyset$;\\        
	\While{$T \neq \emptyset$}{
          \If{$\Delta(T) \leq q+1$}{
		      $\mathcal{P}=\mathcal{P} \cup \{T\}$;\\
                $T=\emptyset$;\\
            }
            \Else{
                $i:=$ the largest index such that there exists a vertex $x_i \in X_i$ with $\deg(x_i) > q+1$;
                
                    ${\mathcal{P}}={\mathcal{P}} \cup \{ T_{x_{i+1}^{q+2}},\ldots , T_{x_{i+1}^{k}}\} \cup \{T'\}$, where $T'$ is a subtree of $T$ induced by $\{x_i\} \cup V(T_{x_{i+1}^1}) \cup \ldots \cup V(T_{x_{i+1}^{q+1}})$;\\
                    $T:=T-T_{x_i}$;
                %}
            }
        
        }

}
}
\end{algorithm}

\begin{thm}\label{thm:algorithm}
For any tree $T$, Algorithm~\ref{al:partition of a tree} returns a smallest partition of $T$ into induced trees of maximum degree $q+1$ and runs in linear time.  
\end{thm}
\begin{proof}
Let $T$ be a tree rooted in an arbitrary vertex $x_0$ that is not a leaf. Let $X_i$ be the set of vertices of $T$ at depth $i$. We will show that a partition ${\mathcal{P}}$ of $T$ into induced trees of maximum degree $q+1$ obtained by Algorithm~\ref{al:partition of a tree} is a smallest such partition. The proof is by induction on the number of vertices of $T$. If $T$ is any tree with $\Delta(T) \leq q+1$, then the algorithm returns $\{T\}$ as the smallest partition into trees of maximum degree at most $q+1$. 

Assume now that $\Delta(T) > q+1$. If $T$ has a support vertex $y$ of degree at most $q+1$, then let $x$ be a leaf neighbor of $y$ and $T'=T-x$. Let $\mathcal{P}$ be the partition of $T$ into induced trees of maximum degree $q+1$ obtained by Algorithm~\ref{al:partition of a tree} and let $T_1\in \mathcal{P}$ be the tree containing $x$. Note that $y \in V(T_1)$, as $\deg(y) \leq q+1$. If $T_2=T_1-x$, then ${\mathcal{P}}'= ({\mathcal{P}}-\{T_1\}) \cup \{T_2\}$ is a partition of $T'$ into induced trees of maximum degree at most $q+1$ obtained by Algorithm~\ref{al:partition of a tree}. By induction hypothesis ${\mathcal{P}}'$ is a smallest such partition. Since for any induced subtree $T^*$ of $T$ it holds that $P_q(T^*) \leq P_q(T)$, it follows that $P_q(T) \leq |{\mathcal{P}}| =|{\mathcal{P}'}|=P_q(T') \leq P_q(T)$, and thus $\mathcal{P}$ is a smallest partition of $T$ into induced trees of maximum degree at most $q+1$.

Assume now that all support vertices of $T$ have degree at least $q+2$. Let $y$ be a support vertex at maximum possible depth.
Let $\mathcal{P}$ be the partition of $T$ into induced trees of maximum degree $q+1$ obtained by Algorithm~\ref{al:partition of a tree}. If $\deg_T(y) = q+2$, let $T'$ be the tree obtained from $T$ by deleting $y$ and all its leaf neighbors, i.e.\ $T'=T-T_y$. If all $q+2$ neighbors of $y$ are leaves, then $T$ is a star and the algorithm is clearly correct. Otherwise, it follows from the algorithm that $T_y \in \mathcal{P}$. Then ${\mathcal{P}}'={\mathcal{P}}-\{T_y\}$ is the partition of $T'$ into induced trees of maximum degree $q+1$ obtained by the algorithm. (Indeed in the first step of the algorithm ${\mathcal{P}}=\{T_y\}$ and $T=T-T_y$.) By the induction hypothesis ${\mathcal{P}}'$ is a smallest such partition, i.e.\ $P_q(T')=|{\mathcal{P}}'|$. Since $P_q(T)=P_q(T')+1$, $\mathcal{P}$ is a smallest such partition of $T$.

Finally, let $\deg_T(y)=k+1 > q+2$. Hence at least $k-q-1 \geq 1$ leaf neighbors of $y$ are singleton trees of $\mathcal{P}$. Let $x$ be one such leaf and let $T'=T-x$. Since ${\mathcal{P}}'={\mathcal{P}}-\{x\}$ is a partition of $T'$ into induced trees of maximum degree $q+1$ obtained by the algorithm, ${\mathcal{P}}'$ is smallest such partition by the induction hypothesis. Since $P_q(T)=P_q(T')+1$ we deduce that $|{\mathcal{P}}|=P_q(T)$, which completes the proof. 

The algorithm is clearly implemented by first implementing BFS order and then proceeding from bottom layers in direction to the root. In this way, the vertices $x_i$ with $\deg(x_i) > q+1$ are located exactly once and thus the algorithm is clearly linear. \qed

\end{proof}

\begin{figure}[ht]
\begin{center}
\begin{tikzpicture}[scale=1,style=thick,x=1cm,y=1cm]
\def\vr{2.5pt} 

\path (0,2) coordinate (0);

\path (0,0) coordinate (a1);
\path (1,0) coordinate (a2);
\path (2,0) coordinate (a3);
\path (3,0) coordinate (a4);
\path (4,0) coordinate (a5);
\path (5,0) coordinate (a6);
\path (1,1) coordinate (b1);
\path (4,1) coordinate (b2);
\path (1,-1) coordinate (c1);
\path (3,-1) coordinate (c2);

\draw (a1)--(b1)--(a2)--(c1)--(c2)--(a4)--(b2)--(a5);
\draw (a3)--(b1)--(b2)--(a6);

\draw (a1) [fill=white] circle (\vr);
\draw (a2) [fill=white] circle (\vr);
\draw (a3) [fill=white] circle (\vr);
\draw (a4) [fill=white] circle (\vr);
\draw (a5) [fill=white] circle (\vr);
\draw (a6) [fill=white] circle (\vr);

\draw (b1) [fill=white] circle (\vr);
\draw (b2) [fill=white] circle (\vr);

\draw (c1) [fill=white] circle (\vr);
\draw (c2) [fill=black] circle (\vr);

\draw[anchor = west] (c2) node {$x$};

\end{tikzpicture}
\end{center}

\caption{Graph $G$ with $F_2(G) < P_2(G)$.}
\label{fig:neenakost Z2 in P2}
\end{figure}
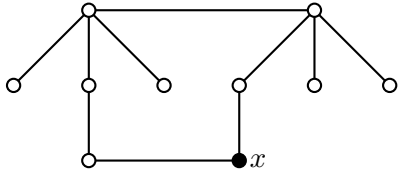

In the article~\cite{AIM}, the authors also showed that the inequality $F_q(G)\geq P_q(G)$ holds for all graphs $G$ when $q=1$. However a similar result for $q \geq 2$ is not true. Let $G$ be a graph as shown in Figure~\ref{fig:neenakost Z2 in P2}. Then $F_2(G)=1$ with $\{x\}$ being 
an $F_2$-set,  while $P_2(G)>1$, since $G$ contains a vertex of degree $4$. 

Next, we turn our attention to the bootstrap percolation number in trees; moreover, our goal is to determine $\sigma_{(p,q)}$ when $p>1$. Riedl described an algorithm for finding a smallest $p$-percolating set of a tree $T$ in~\cite{algoritem}. The algorithm works by reducing the tree $T$ into smaller subtrees $T_i$ until every vertex is either a leaf or of an appropriate degree. Then one can arbitrarily choose a $p$-percolating set of the subtree $T_i$ by selecting all of its leaves. Notice that leaves always satisfy the condition of a $(p,1)$ spreading set. When combining the percolating sets of trees $T_i$ and $T_j$ to form the percolating set of a larger tree, we do this by gluing the trees back together in a single vertex $v$. As it turns out, $v$ always acts as a leaf, when proceeding with the infection process, therefore this set also satisfies the condition of a $(p,1)$-spreading set. All in all, the algorithm from~\cite{algoritem} works also under the conditions of $(p,1)$-spreading. We summarize this in the following result.

\begin{thm}\label{thm:(p,1) na drevesih}
If $p \geq 2$ and $T$ is a tree with maximum degree $\Delta$, then $\sigma_{(p,q)}(T)= \sigma_{(p,1)}(T)$ for any $q\in \mathbb{N}\cup\{\infty\}$. In particular, $\sigma_{(p,\infty)}(T)=\sigma_{(p,\Delta)}(T) = \sigma_{(p,1)}(T)$. 
\end{thm}

Riedl~\cite{algoritem} also proved that the algorithm runs in $O(n)$ time, which is thus the complexity of determining $\sigma_{(p,q)}(T)$ for an arbitrary $p\ge 2$, $q\in \mathbb{N}$, and an arbitrary tree $T$. 

We conclude this subsection by compressing all of its theorems into the following result.

\begin{cor}
If $T$ is a tree, $p\in \mathbb{N}$ and $q\in \mathbb{N}\cup\{\infty\}$, then $\spq(T)$ can be determined in linear time. 
\end{cor}

%%%%%%%%%
\subsection{Bounds and extremal families}
%%%%%%%%%%%

We start by observing that for every $q\in \mathbb{N}\cup \{\infty\}$, we have
\begin{displaymath}
\spq(K_{1,n-1})=
    \begin{cases}
        n-1 & \text{if } 2\leq p\leq n-1\\
        n & \text{if } p\geq n.
    \end{cases}
\end{displaymath}

Next, we present a general upper bound on the $(p,q)$-spreading number of trees, and determine the extremal trees.

\begin{thm}\label{extremal_tree_max}
Let $T$ be a tree of order $n\geq 5$ and $q\in \mathbb{N}\cup\{\infty\}$.
\begin{itemize}
    \item[i)] If $p=2$, then $\sigma_{(p,q)}(T)\leq n-1$ with equality if and only if $T=K_{1,n-1}$;
    \item [ii)] If $p\geq 3$, then $\sigma_{(p,q)}(T)\leq n$ with equality if and only if $\Delta(T)<p$.
\end{itemize}
\end{thm}
\begin{proof}
    \noindent i) If $v$ is a vertex with $\deg (v)\geq 2$, then $V(T)\setminus \{v\}$ is a $(p,q)$-spreading set of $T$. So, it is clear that $\sigma_{(p,q)}(T)\leq n-1$. If $T=P_n$, then $\sigma_{(p,q)}(P_n)=\lceil(n+1)/2\rceil<n-1$ for $n\geq 5$. If $T\neq P_n$ and $T\neq K_{1,n-1}$, then there exist vertices $u$ and $v$ with $\deg (u)\geq 3$ and $\deg (v)\geq 2$. Now, $V(T)\setminus \{u,v\}$ is a $(p,q)$-spreading set and hence $\sigma_{(p,q)}(T)\leq n-2$.
    \noindent ii) It is clear that $\sigma_{(p,q)}(T)\leq n$, and by Remark~\ref{r:basic}(2), if $\Delta(T)<p$ then $\sigma_{(p,q)}(T)=n$. Suppose that $\Delta (T)\geq p$ and let $v$ be a vertex with $\deg (v)\geq p$. Now, $V(T)\setminus \{v\}$ is a $(p,q)$-spreading set and hence $\sigma_{(p,q)}(T)\leq n-1$.
\end{proof}
\medskip

By the lower bound from~\cite{algoritem} and noting that $\sigma_{(p,1)}(T)=\spq(T)$ for any tree $T$ we immediately infer the following.
\begin{thm}\label{thm:tree_lower}
    Let $T$ be a tree of order $n$ and $p\geq 2$. Then,
    $$\sigma_{(p,q)}(T)\geq \Big\lceil \frac{(p-1)}{p}n+\frac{1}{p}\Big\rceil.$$
\end{thm}

\iffalse
\begin{proof}
    We will prove that $p|S|>(p-1)n$ for every $(p,1)$-spreading set $S$ of $T$.  We proceed by induction on $n$. The result is clear for $n\leq 2$. Suppose that $T$ is a rooted tree with $n\geq 3$ and let $u$ be a support vertex of maximum distance from the root. Let $v_1,\dots , v_l$ be the leaf neighbors of $u$ and let $u'$ be the unique non-leaf neighbor of $u$. Since $p\geq 2$, all of the leaves $v_1,\dots , v_l$ are contained in $S$. We will consider two cases.

    {\it Case 1:} $u$ is in $S$. In this case $S\setminus \{v_1\}$ is a $(p,1)$-spreading set of $T-v_1$. By the induction hypothesis on $T-v_1$, we have $$p\,(|S|-1)>(p-1)(n-1)$$
    which is equivalent to $p|S|>1+(p-1)n$.

{\it Case 2:} $u$ is not in $S$. If $u$ is not in $S$, then $l\geq p-1$. Now we will consider two sub-cases based on $p$ and $l$.

{\it Case 2.1:} $l\geq p$. $(S\setminus\{v_1,\dots , v_l\})\cup \{u\}$ is a spreading set of $T\setminus \{v_1,\dots , v_l\}$ and by the induction hypothesis on $T\setminus \{v_1,\dots , v_l\}$, we have
$$p(|S|-l+1)>(p-1)(n-l)$$
which is equivalent to $p|S|+p>(p-1)n+l$. The latter implies that $p|S|>(p-1)n$, as $l\geq p$.

{\it Case 2.2:} $l= p-1$. The vertex $u$ cannot be infected before the vertex $u'$ in any $(p,1)$-spreading process. Therefore, $S\setminus \{u,v_1,\dots ,v_l\}$ is a $(p,1)$-spreading set of $T\setminus \{u,v_1,\dots,v_l\}$. By the induction hypothesis, 
$$p(|S|-l)>(p-1)(n-l-1)$$
which is equivalent to $p|S|>(p-1)n$ as $l=p-1$.

\end{proof}
\fi

Let $\operatorname{rem}(N,p)$ denote the remainder when $N$ is divided by $p$ and  let $f(n,p)=\lceil ((p-1)n+1)/p\rceil$. We say that a tree $T$ of order $n$ satisfies property $\mathcal{P}(n,p)$ if $V=V(T)$ can be partitioned into two subsets $S$ and $V\setminus S$ such that $|S|=f(n,p)$ and the vertices of $V\setminus S$ can be ordered as $v_1,\dots , v_{n-f(n,p)}$ and they satisfy the following:
\begin{itemize}
    \item[i)] $|N_{F_i}(v_{i})|\geq p$ for all $i$, and
    \item[ii)]  $\operatorname{rem}(n-1,p)\geq \sum_{i=1}^{n-f(n,p)}\left(|N_{F_{i}}(v_i)|-p\right)$, and
    \item[iii)] There are exactly $\operatorname{rem}(n-1,p)-\sum_{i=1}^{n-f(n,p)}\left(|N_{F_{i}}(v_i)|-p\right)$ edges in the subgraph induced by $S$,
\end{itemize}
where each $F_i$ is a subforest of $T$ defined recursively by 
$$F_1=\{v_1\}\cup S_1\ \text{and} \ F_{i+1}=F_i\cup \{v_{i+1}\} \cup S_{i+1}\ \text{for}\ 1\leq i\leq n-f(n,p)-1$$ and $S_i=\{u\in S\setminus \left(\bigcup\limits_{j<i}S_j\right):\ \text{there is a path in}\ T[S] \ \text{from}\ u\ \text{to a vertex in}\ N_S(v_i) \}$.

 There are many trees satisfying the property $\mathcal{P}(n,p)$. As an example, construct a tree $T^*$ as follows: let both $S=\{w_1,\dots , w_{|S|}\}$ and $V(T)\setminus S=\{v_1,\dots , v_{n-|S|}\}$ be independent sets with $|S|=f(n,p)$, and for $1\leq i \leq n-|S|-1$, $N(v_i)=\{w_{1+(i-1)(p-1)},\dots , w_{1+i(p-1)}\}$ and $N(v_{n-|S|})=\{w_{1+(n-|S|-1)(p-1)},\dots , w_{|S|}\}$, see Figure~\ref{fig:treeT*}. Observe that if $n-1$ is divisible by $p$, then $S$ is necessarily an independent set.

\begin{figure}[ht]
\begin{center}
\begin{multicols}{2}
\begin{tikzpicture}[scale=.75,style=thick,x=1cm,y=1cm]
\def\vr{2.5pt} % \vr = vertex radius;
% define vertices

\path (0,2) coordinate (0);

\path (0,0) coordinate (a1);
\path (1,0) coordinate (a2);
\path (2,0) coordinate (a3);
\path (3,0) coordinate (a4);
\path (4,0) coordinate (a5);
\path (5,0) coordinate (a6);
\path (6,0) coordinate (a7);
\path (7,0) coordinate (a8);
\path (1,1) coordinate (b1);
\path (3,1) coordinate (b2);
\path (3,1.5) coordinate (x2);
\path (5,1) coordinate (b3);

\draw (a1)--(b1)--(a3)--(b2)--(a5)--(b3)--(a7);
\draw (a2)--(b1);
\draw (a4)--(b2);
\draw (a6)--(b3);
\draw (a8)--(b3);

\draw (a1) [fill=black] circle (\vr);
\draw (a2) [fill=black] circle (\vr);
\draw (a3) [fill=black] circle (\vr);
\draw (a4) [fill=black] circle (\vr);
\draw (a5) [fill=black] circle (\vr);
\draw (a6) [fill=black] circle (\vr);
\draw (a7) [fill=black] circle (\vr);
\draw (a8) [fill=black] circle (\vr);

\draw (b1) [fill=white] circle (\vr);
\draw (b2) [fill=white] circle (\vr);
\draw (b3) [fill=white] circle (\vr);

\draw[anchor = east] (b1) node {$v_1$};
\draw[anchor = east] (b2) node {$v_2$};
\draw[anchor = south] (x2) node {$T^*$};
\draw[anchor = east] (b3) node {$v_3$};

\end{tikzpicture}

\columnbreak

\begin{tikzpicture}[scale=.75,style=thick,x=1cm,y=1cm]
\def\vr{2.5pt} % \vr = vertex radius;
% define vertices

\path (0,2) coordinate (0);

\path (0,0) coordinate (a1);
\path (1,0) coordinate (a2);
\path (2,0) coordinate (a3);
\path (3,0) coordinate (a4);
\path (4,0) coordinate (a5);
\path (5,0) coordinate (a6);
\path (6,0) coordinate (a7);
\path (7,0) coordinate (a8);
\path (1,1) coordinate (b1);
\path (3,1) coordinate (b2);
\path (3,1.5) coordinate (y2);
\path (5,1) coordinate (b3);

\draw (a1)--(b1)--(b2)--(a5)--(b3)--(a7);
\draw (b1)--(a3);
\draw (a2)--(b1);
\draw (a4)--(b2);
\draw (a6)--(b3);
\draw (a8)--(a7);

\draw (a1) [fill=black] circle (\vr);
\draw (a2) [fill=black] circle (\vr);
\draw (a3) [fill=black] circle (\vr);
\draw (a4) [fill=black] circle (\vr);
\draw (a5) [fill=black] circle (\vr);
\draw (a6) [fill=black] circle (\vr);
\draw (a7) [fill=black] circle (\vr);
\draw (a8) [fill=black] circle (\vr);

\draw (b1) [fill=white] circle (\vr);
\draw (b2) [fill=white] circle (\vr);
\draw (b3) [fill=white] circle (\vr);

\draw[anchor = east] (b1) node {$v_1$};
\draw[anchor = west] (b2) node {$v_2$};
\draw[anchor = west] (b3) node {$v_3$};
\draw[anchor = south] (y2) node {$T'$};

\end{tikzpicture}

\end{multicols}

\end{center}

\caption{Trees $T^*$ (left) and $T'$ (right) with $n=11$ and $p=3$.}
\label{fig:treeT*}
\end{figure}
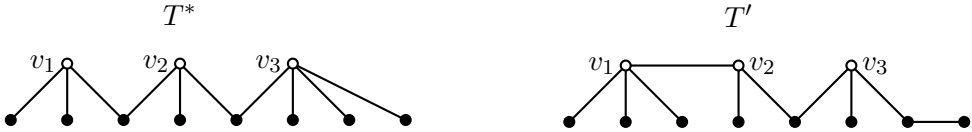

Next, we characterize all extremal trees in Theorem~\ref{thm:tree_lower} via property $\mathcal{P}(n,p)$. Given a spreading set $S$ of a graph $G$, we say that $(v_1,\dots , v_{n-|S|})$ is a $(p,q)$-{\it spreading sequence} if $V(G)\setminus S=\{v_1,\dots , v_{n-|S|}\}$ and there is a spreading process in which vertices can become blue in the order of the sequence. Note that spreading sequence is not necessarily unique. For example, in the tree $T^*$ in Figure~\ref{fig:treeT*}, any ordering of vertices in $\{v_1,v_2,v_3\}$ is a spreading sequence, and the tree $T'$ has exactly three spreading sequences which are $(v_1,v_2,v_3)$, $(v_1,v_3,v_2)$ and $(v_3,v_1,v_2)$. 
\begin{thm}
    Let $T$ be a tree of order $n$ and $p\geq 2$. Then, $\sigma_{(p,q)}(T)=f(n,p)$ if and only if $T$ satisfies property $\mathcal{P}(n,p)$.
\end{thm}

\begin{proof}
    First, suppose that $T$ satisfies property $\mathcal{P}(n,p)$. Then, $(v_1,\dots ,v_{n-f(n,p)})$ is a $(p,q)$-spreading sequence of $T$ since $|N_{F_i}(v_{i})|\geq p$ for all $i$. Thus, $S$ is a $(p,q)$-spreading set of $T$ and $\sigma_{(p,q)}(T)\leq f(n,p)$. By Theorem~\ref{thm:tree_lower}, we have $\sigma_{(p,q)}(T)\geq f(n,p)$. Hence, we obtain $\sigma_{(p,q)}(T)=f(n,p)$.

    Now, let $S$ be a $(p,q)$-spreading set of $T$ and $(v_1,\dots , v_{n-|S|})$ be a $(p,q)$-spreading sequence of $T$.  At step $t$ of the spreading process, all vertices in $F_t$ are blue and it is clear that $|N_{F_i}(v_{i})|\geq p$ for all $i$, as $(v_1,\dots , v_{n-|S|})$ is a $(p,q)$-spreading sequence of $T$. Thus $i)$ holds. Let $k_t$ be the number of edges in the subgraph induced by $S\cap V(F_t)$ and $c_t$ be the number of connected components of $F_t$. Now, we claim that for all $1\leq t\leq n-|S|$,
    \begin{equation}\label{eq:property_tree}
        |V(F_t)|=k_t+c_t+\sum_{i=1}^t|N_{F_i}(v_i)|.
    \end{equation}
To prove \eqref{eq:property_tree}, we proceed by induction on $t$. For $t=1$, we have $|V(F_1)|=1+|S_1|$, $k_1=|S_1|-|N_S(v_1)|$, $c_1=1$ and $|N_{F_1}(v_1)|=|N_S(v_1)|$. Hence, we obtain \eqref{eq:property_tree}. Now suppose that $t>1$. By the induction hypothesis, $|V(F_{t-1})|=k_{t-1}+c_{t-1}+\sum_{i=1}^{t-1}|N_{F_i}(v_i)|$. Since $T$ is a tree, the vertex $v_t$ has at most one neighbor in each component of $F_{t-1}$. Moreover, $V(F_{t-1})\cap S_t=\emptyset$ and there is no edge from vertices of $S_t$ to vertices of $V(F_{t-1})$ by the definition. Hence, we obtain that $k_t=k_{t-1}+|S_t|-|N_{F_t}(v_t)|+|V(F_{t-1})\cap N_{F_t}(v_t)|$ and $c_t=1+c_{t-1}-|V(F_{t-1})\cap N_{F_t}(v_t)|$.

\begin{eqnarray*}
    |V(F_{t})|&=& 1+|S_t|+ |V(F_{t-1})|\\
    &=& 1+|S_t|+k_{t-1}+c_{t-1}+\sum_{i=1}^{t-1}|N_{F_i}(v_i)|\\
    &=& 1+c_{t-1}-|V(F_{t-1})\cap N_{F_t}(v_t)|+k_t+\sum_{i=1}^{t}|N_{F_i}(v_i)|\\
     &=& c_t+k_t+\sum_{i=1}^{t}|N_{F_i}(v_i)|\\
\end{eqnarray*}

 For $t=n-|S|$, we have $F_{n-|S|}=T$ and therefore, $c_{n-|S|}=1$ and $k_{n-|S|}=|E(G_S)|$. So, \eqref{eq:property_tree} becomes    
\begin{equation}\label{eq:extremal_tree_2}
    n-1=p\,(n-|S|)+|E(T[S])|+\sum_{i=1}^{n-|S|}(|N_{F_i}(v_i)|-p).
\end{equation}

\noindent It is easy to verify that $n-1=p\,(n-f(n,p))+\operatorname{rem}(n-1,p)$ for all $n$ and $p$. Hence, if $|S|=f(n,p)$, then $\operatorname{rem}(n-1,p)=|E(T[S])|+\sum_{i=1}^{n-|S|}(|N_{F_i}(v_i)|-p)$. Thus, every tree $T$ with a spreading set $S$ of size $f(n,p)$ satisfies (ii) and (iii) in the definition of property $\mathcal{P}(n,p)$.

 \end{proof}

%%%%%%%%%%%%%%%%%%%%%%%%%%%%%%%%
\section{Spreading in grids}
\label{sec:grids}
%%%%%%%%%%%%%%%%%%%%%%%%%%%%%%%%
In this section, we determine the $(p,q)$-spreading number of grids for all possible pairs $(p,q)$ with the exception when $p=3$. Let $G=P_m \square P_n$, and without loss of generality let $n \leq m$. 

Let $p=1$, that is consider the $q$-forcing problem in grids. If $q=1$, that is, considering the zero forcing problem, it was proved in~\cite{AIM} that $Z(G)=F_1(G)=\sigma_{(1,1)}(G)=n$. Clearly, we also have $F_q(G)=1$ for $q \geq 2$. 

The case when $p=3$ and $q \geq 4$, that is, $\spq(G)=m(G,3)$, was recently studied in~\cite{HH-2023+, romer}, but it seems that this is a difficult problem for which only partial results are known. This is the only case of grids that we will not resolve.

If $p > \Delta(G)=4$, Remark~\ref{r:basic}(2) implies that $\spq(G)=|V(G)|=mn$. The case $p=4$ is resolved in the following theorem.

\begin{thm}
If $m,n\ge 3$ are positive integers, then $\sigma_{(4,q)}(P_m \cp P_n)= 2m+2n-4 + \lfloor \frac{(m-2)(n-2)}{2} \rfloor$ for all $q \in \mathbb{N}$.
\end{thm}

\begin{proof}
Let $V(P_m)=\lbrace 1,2,\ldots ,m \rbrace$ and $V(P_n)=\lbrace 1,2,\ldots ,n \rbrace$ and let $S$ be a $(4,q)$-spreading set of $P_m \cp P_n$. (Let us call vertices $x\in P_m \cp P_n$ with $\deg(x)<4$ {\em boundary vertices}.) By Remark~\ref{r:basic}(1), all $2m+2n-4$ boundary vertices of $P_m \Box P_n$ are contained in $S$. From the remaining $(m-2) \times (n-2)$ grid, we claim that we must select vertices that form a vertex cover. Indeed, suppose that there exists an edge $uv$ such that $u \notin S$ and $v \notin S$. Then in any step of the color changing procedure, $u$ and $v$ each have at most three blue neighbors, and thus $u$ and $v$ cannot be colored blue, and so $S$ is not a $(4,q)$-spreading set, a contradiction. As the minimum vertex cover number of the $(m-2) \times (n-2)$ grid is $\lfloor \frac{(m-2)(n-2)}{2} \rfloor$, we get $|S| \geq 2m+2n-4 + \lfloor \frac{(m-2)(n-2)}{2} \rfloor$. For the converse, one can easily check that the boundary vertices of $P_m \Box P_n$ together with the minimum vertex cover set of the remaining $(m-2)\times (n-2)$ grid forms a $(4,q)$-spreading set and thus $\sigma_{(4,q)}(P_m \cp P_n) \leq 2m+2n-4 + \lfloor \frac{(m-2)(n-2)}{2} \rfloor$. \qed
 
\end{proof}

\begin{figure}[ht!]
\begin{center}
\begin{tikzpicture}[scale=0.8,style=thick,x=1cm,y=1cm]
\def\vr{2.5pt} % \vr = vertex radius;
% define vertices

\path (-1,1) coordinate (a);
\path (-1,2) coordinate (b);
\path (-1,3) coordinate (c);
\path (-1,4) coordinate (d);
\path (-1,5) coordinate (e);
\path (0,0.2) coordinate (1);
\path (1,0.2) coordinate (2);
\path (2,0.2) coordinate (3);
\path (3,0.2) coordinate (4);
\path (4,0.2) coordinate (5);
\path (5,0.2) coordinate (6);
\path (6,0.2) coordinate (7);
\path (7,0.2) coordinate (8);
\path (8,0.2) coordinate (9);
\path (9,0.2) coordinate (10);

\path (0,1) coordinate (a1);
\path (0,2) coordinate (b1);
\path (0,3) coordinate (c1);
\path (0,4) coordinate (d1);
\path (0,5) coordinate (e1);

\path (1,1) coordinate (a2);
\path (1,2) coordinate (b2);
\path (1,3) coordinate (c2);
\path (1,4) coordinate (d2);
\path (1,5) coordinate (e2);

\path (2,1) coordinate (a3);
\path (2,2) coordinate (b3);
\path (2,3) coordinate (c3);
\path (2,4) coordinate (d3);
\path (2,5) coordinate (e3);

\path (3,1) coordinate (a4);
\path (3,2) coordinate (b4);
\path (3,3) coordinate (c4);
\path (3,4) coordinate (d4);
\path (3,5) coordinate (e4);

\path (4,1) coordinate (a5);
\path (4,2) coordinate (b5);
\path (4,3) coordinate (c5);
\path (4,4) coordinate (d5);
\path (4,5) coordinate (e5);

\path (5,1) coordinate (a6);
\path (5,2) coordinate (b6);
\path (5,3) coordinate (c6);
\path (5,4) coordinate (d6);
\path (5,5) coordinate (e6);

\path (6,1) coordinate (a7);
\path (6,2) coordinate (b7);
\path (6,3) coordinate (c7);
\path (6,4) coordinate (d7);
\path (6,5) coordinate (e7);

\path (7,1) coordinate (a8);
\path (7,2) coordinate (b8);
\path (7,3) coordinate (c8);
\path (7,4) coordinate (d8);
\path (7,5) coordinate (e8);

\path (8,1) coordinate (a9);
\path (8,2) coordinate (b9);
\path (8,3) coordinate (c9);
\path (8,4) coordinate (d9);
\path (8,5) coordinate (e9);

\path (9,1) coordinate (a10);
\path (9,2) coordinate (b10);
\path (9,3) coordinate (c10);
\path (9,4) coordinate (d10);
\path (9,5) coordinate (e10);

%  edges
\draw (a1)--(b1)--(c1)--(d1)--(e1);
\draw (a2)--(b2)--(c2)--(d2)--(e2);
\draw (a3)--(b3)--(c3)--(d3)--(e3);
\draw (a4)--(b4)--(c4)--(d4)--(e4);
\draw (a5)--(b5)--(c5)--(d5)--(e5);
\draw (a6)--(b6)--(c6)--(d6)--(e6);
\draw (a7)--(b7)--(c7)--(d7)--(e7);
\draw (a8)--(b8)--(c8)--(d8)--(e8);
\draw (a9)--(b9)--(c9)--(d9)--(e9);
\draw (a10)--(b10)--(c10)--(d10)--(e10);

\draw (a1)--(a2)--(a3)--(a4)--(a5)--(a6)--(a7)--(a8)--(a9)--(a10);
\draw (b1)--(b2)--(b3)--(b4)--(b5)--(b6)--(b7)--(b8)--(b9)--(b10);
\draw (c1)--(c2)--(c3)--(c4)--(c5)--(c6)--(c7)--(c8)--(c9)--(c10);
\draw (d1)--(d2)--(d3)--(d4)--(d5)--(d6)--(d7)--(d8)--(d9)--(d10);
\draw (e1)--(e2)--(e3)--(e4)--(e5)--(e6)--(e7)--(e8)--(e9)--(e10);

% vertices
\draw (a1) [fill=black] circle (\vr);
\draw (b1) [fill=white] circle (\vr);
\draw (c1) [fill=white] circle (\vr);
\draw (d1) [fill=white] circle (\vr);
\draw (e1) [fill=white] circle (\vr);

\draw (a2) [fill=white] circle (\vr);
\draw (b2) [fill=black] circle (\vr);
\draw (c2) [fill=white] circle (\vr);
\draw (d2) [fill=white] circle (\vr);
\draw (e2) [fill=white] circle (\vr);

\draw (a3) [fill=white] circle (\vr);
\draw (b3) [fill=white] circle (\vr);
\draw (c3) [fill=black] circle (\vr);
\draw (d3) [fill=white] circle (\vr);
\draw (e3) [fill=white] circle (\vr);

\draw (a4) [fill=white] circle (\vr);
\draw (b4) [fill=white] circle (\vr);
\draw (c4) [fill=white] circle (\vr);
\draw (d4) [fill=black] circle (\vr);
\draw (e4) [fill=white] circle (\vr);

\draw (a5) [fill=white] circle (\vr);
\draw (b5) [fill=white] circle (\vr);
\draw (c5) [fill=white] circle (\vr);
\draw (d5) [fill=white] circle (\vr);
\draw (e5) [fill=black] circle (\vr);

\draw (a6) [fill=white] circle (\vr);
\draw (b6) [fill=white] circle (\vr);
\draw (c6) [fill=white] circle (\vr);
\draw (d6) [fill=white] circle (\vr);
\draw (e6) [fill=white] circle (\vr);

\draw (a7) [fill=white] circle (\vr);
\draw (b7) [fill=white] circle (\vr);
\draw (c7) [fill=white] circle (\vr);
\draw (d7) [fill=white] circle (\vr);
\draw (e7) [fill=black] circle (\vr);

\draw (a8) [fill=white] circle (\vr);
\draw (b8) [fill=white] circle (\vr);
\draw (c8) [fill=white] circle (\vr);
\draw (d8) [fill=white] circle (\vr);
\draw (e8) [fill=white] circle (\vr);

\draw (a9) [fill=white] circle (\vr);
\draw (b9) [fill=white] circle (\vr);
\draw (c9) [fill=white] circle (\vr);
\draw (d9) [fill=white] circle (\vr);
\draw (e9) [fill=black] circle (\vr);

\draw (a10) [fill=white] circle (\vr);
\draw (b10) [fill=white] circle (\vr);
\draw (c10) [fill=white] circle (\vr);
\draw (d10) [fill=white] circle (\vr);
\draw (e10) [fill=black] circle (\vr);

% text
\draw[anchor = east] (a) node {1};
\draw[anchor = east] (b) node {2};
\draw[anchor = east] (c) node {3};
\draw[anchor = east] (d) node {4};
\draw[anchor = east] (e) node {5};
\draw[anchor = north] (1) node {1};
\draw[anchor = north] (2) node {2};
\draw[anchor = north] (3) node {3};
\draw[anchor = north] (4) node {4};
\draw[anchor = north] (5) node {5};
\draw[anchor = north] (6) node {6};
\draw[anchor = north] (7) node {7};
\draw[anchor = north] (8) node {8};
\draw[anchor = north] (9) node {9};
\draw[anchor = north] (10) node {10};

\end{tikzpicture}
\end{center}

\caption{A minimum $(2,4)$-spreading set $S$ for $G=P_{10} \cp P_5$.}
\label{fig:(2,4)-spreading set}
\end{figure}
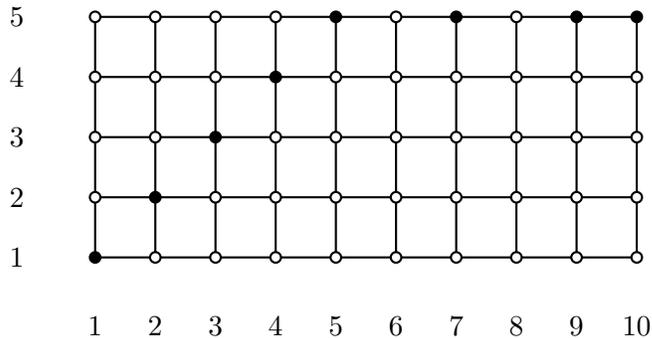

\bigskip 

Next, we consider the case $p=2$.
Let $G=P_m \square P_n$, and without loss of generality, suppose that $n \leq m$. 
Bootstrap percolation was already studied in grids by Bollob{\' a}s, who proved that $m(G,2)=\sigma_{(2,\infty)}(G) = \lceil \frac{n+m}{2} \rceil$~\cite{Bol-2006}. Since $\Delta (G) = 4$ it follows from Remark~\ref{r:basic}(3) that $\sigma_{(2,q)}(G) =\sigma_{(2,\infty)}(G) = \lceil \frac{n+m}{2} \rceil$ for all $q \geq 4$. In his solution Bollob\' as proposed a $(2, \infty)$-spreading set (2-bootstrap percolating set) which consists of $n$ diagonal vertices on the first $n$ columns and in the remaining columns $\lceil \frac{m-n}{2} \rceil$ vertices, one from every other column and one from the last column (see Figure~\ref{fig:(2,4)-spreading set}). Notice that this set is also a $(2,3)$- and a $(2,2)$-spreading set. Therefore Remark~\ref{r:basic}(5) implies that $\lceil \frac{n+m}{2} \rceil = \sigma_{(2,\infty)}(G) \leq \sigma_{(2,3)}(G) \leq \sigma_{(2,2)}(G) \leq  \lceil \frac{n+m}{2} \rceil$. Thus the only unsolved case for $p=2$ is the $(2,1)$-spreading number of $G$, which we consider in the next result.

\begin{thm}
If $m,n\ge 3$ are positive integers, then $\sigma_{(2,1)}(P_m\cp P_n) = \lceil \frac{n+m+1}{2} \rceil$.
\end{thm}

\begin{proof}
First let $n+m$ be odd and without loss of generality let $m$ be even and $n$ odd. Note that $\lceil \frac{n+m+1}{2} \rceil = \lceil \frac{n+m}{2} \rceil$, therefore $\lceil \frac{n+m}{2} \rceil = \sigma_{(2,\infty)}(G) \leq \sigma_{(2,1)}(G)$. Thus it remains to find a $(2,1)$-spreading set of the same size. Consider the set $S= \lbrace (1,1),(2,1),(4,1), \dots ,(m,1),(1,3),(1,5), \dots ,(1,n)\rbrace$ of size $\frac{n+1}{2} + \frac{m}{2} = \lceil \frac{n+m}{2} \rceil$ (see Figure~\ref{fig:(2,1)-spreading set}). It is not hard to verify that $S$ is a $(2,1)$-spreading set of $G$. For instance, by using the color change rule one can find the following sequence of color changes: $$(1,2),(2,2),(3,1),(3,2),(2,3),(1,4),(2,4),(3,3),(4,2),(5,1),(5,2),\dots $$

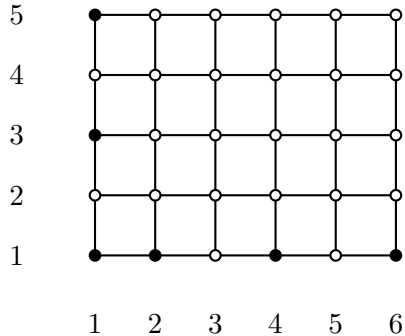
\begin{figure}[ht!]
\begin{center}
\begin{tikzpicture}[scale=0.8,style=thick,x=1cm,y=1cm]
\def\vr{2.5pt} % \vr = vertex radius;
% define vertices

\path (-1,1) coordinate (a);
\path (-1,2) coordinate (b);
\path (-1,3) coordinate (c);
\path (-1,4) coordinate (d);
\path (-1,5) coordinate (e);
\path (0,0.2) coordinate (1);
\path (1,0.2) coordinate (2);
\path (2,0.2) coordinate (3);
\path (3,0.2) coordinate (4);
\path (4,0.2) coordinate (5);
\path (5,0.2) coordinate (6);

\path (0,1) coordinate (a1);
\path (0,2) coordinate (b1);
\path (0,3) coordinate (c1);
\path (0,4) coordinate (d1);
\path (0,5) coordinate (e1);

\path (1,1) coordinate (a2);
\path (1,2) coordinate (b2);
\path (1,3) coordinate (c2);
\path (1,4) coordinate (d2);
\path (1,5) coordinate (e2);

\path (2,1) coordinate (a3);
\path (2,2) coordinate (b3);
\path (2,3) coordinate (c3);
\path (2,4) coordinate (d3);
\path (2,5) coordinate (e3);

\path (3,1) coordinate (a4);
\path (3,2) coordinate (b4);
\path (3,3) coordinate (c4);
\path (3,4) coordinate (d4);
\path (3,5) coordinate (e4);

\path (4,1) coordinate (a5);
\path (4,2) coordinate (b5);
\path (4,3) coordinate (c5);
\path (4,4) coordinate (d5);
\path (4,5) coordinate (e5);

\path (5,1) coordinate (a6);
\path (5,2) coordinate (b6);
\path (5,3) coordinate (c6);
\path (5,4) coordinate (d6);
\path (5,5) coordinate (e6);

%  edges
\draw (a1)--(b1)--(c1)--(d1)--(e1);
\draw (a2)--(b2)--(c2)--(d2)--(e2);
\draw (a3)--(b3)--(c3)--(d3)--(e3);
\draw (a4)--(b4)--(c4)--(d4)--(e4);
\draw (a5)--(b5)--(c5)--(d5)--(e5);
\draw (a6)--(b6)--(c6)--(d6)--(e6);

\draw (a1)--(a2)--(a3)--(a4)--(a5)--(a6);
\draw (b1)--(b2)--(b3)--(b4)--(b5)--(b6);
\draw (c1)--(c2)--(c3)--(c4)--(c5)--(c6);
\draw (d1)--(d2)--(d3)--(d4)--(d5)--(d6);
\draw (e1)--(e2)--(e3)--(e4)--(e5)--(e6);

% vertices
\draw (a1) [fill=black] circle (\vr);
\draw (b1) [fill=white] circle (\vr);
\draw (c1) [fill=black] circle (\vr);
\draw (d1) [fill=white] circle (\vr);
\draw (e1) [fill=black] circle (\vr);

\draw (a2) [fill=black] circle (\vr);
\draw (b2) [fill=white] circle (\vr);
\draw (c2) [fill=white] circle (\vr);
\draw (d2) [fill=white] circle (\vr);
\draw (e2) [fill=white] circle (\vr);

\draw (a3) [fill=white] circle (\vr);
\draw (b3) [fill=white] circle (\vr);
\draw (c3) [fill=white] circle (\vr);
\draw (d3) [fill=white] circle (\vr);
\draw (e3) [fill=white] circle (\vr);

\draw (a4) [fill=black] circle (\vr);
\draw (b4) [fill=white] circle (\vr);
\draw (c4) [fill=white] circle (\vr);
\draw (d4) [fill=white] circle (\vr);
\draw (e4) [fill=white] circle (\vr);

\draw (a5) [fill=white] circle (\vr);
\draw (b5) [fill=white] circle (\vr);
\draw (c5) [fill=white] circle (\vr);
\draw (d5) [fill=white] circle (\vr);
\draw (e5) [fill=white] circle (\vr);

\draw (a6) [fill=black] circle (\vr);
\draw (b6) [fill=white] circle (\vr);
\draw (c6) [fill=white] circle (\vr);
\draw (d6) [fill=white] circle (\vr);
\draw (e6) [fill=white] circle (\vr);

% text
\draw[anchor = east] (a) node {1};
\draw[anchor = east] (b) node {2};
\draw[anchor = east] (c) node {3};
\draw[anchor = east] (d) node {4};
\draw[anchor = east] (e) node {5};
\draw[anchor = north] (1) node {1};
\draw[anchor = north] (2) node {2};
\draw[anchor = north] (3) node {3};
\draw[anchor = north] (4) node {4};
\draw[anchor = north] (5) node {5};
\draw[anchor = north] (6) node {6};

\end{tikzpicture}
\end{center}

\caption{A minimum $(2,1)$-spreading set $S$ for $G=P_6 \cp P_5$.}
\label{fig:(2,1)-spreading set}
\end{figure}

Now let $n+m$ be even and let $S$ be an arbitrary $(2,1)$-spreading set. Suppose $|S| \leq \frac{n+m}{2}$. Firstly, according to Lemma \ref{lema: delta > q}, set $S$ must contain a pair of adjacent vertices. Now we use Bollob\' as' idea of replacing the grid with an uncolored chess board where each (unit) square corresponds to a vertex in $G$. Color each square corresponding to a vertex of $S$ with blue, others with white. Whenever a white square (adjacent to at least two blue squares) is colored blue, the total perimeter of the blue domain (sum of perimeters of blue polygons consisting of connected blue squares) does not increase. Since the perimeter of the whole chess board (every vertex/square is blue) is $2(n+m)$, the starting perimeter must also be at least $2(n+m)$. Every blue square contributes at most 4 units to the total perimeter. Since we know that at least 2 blue squares must be adjacent, their total perimeter can be at most 6. The remaining $\frac{n+m}{2} - 2$ blue squares contribute at most $2(n+m)-8$, which means the total perimeter is at most $2(n+m)-2$, a contradiction, hence $|S| \geq \frac{n+m}{2} +1$.

For the upper bound we consider two cases. If $n$ and $m$ are odd, let $S = \lbrace (1,1),(2,1)$, $(4,1),(6,1),\dots ,(m-3,1),(m-1,1),(m,1),(1,3),(1,5),\dots ,(1,n) \rbrace$. 
If $n$ and $m$ are even then let $S = \lbrace (1,1),(2,1),(4,1),(6,1)$, $\dots , (m,1),(1,3),(1,5),\dots$, $(1,n-3),(1,n-1)(1,n)\rbrace$. 
In either case, $S$ contains exactly $\frac{n+m}{2}+1$ vertices. One can easily check that a similar pattern of blue color spreading as described in the case of $n+m$ being odd proves that $S$ is a $(2,1)$-spreading set. \qed
\end{proof}

All known results about the $(p,q)$-spreading number on $m \times n$ grids are condensed in Table~\ref{Tabela1}. Note that the values in the third row are still open, while all other values of $\sigma_{(p,q)}(P_m\cp P_n)$ have been established.

\begin{table}[ht!]
\begin{center}
\begin{tabular}{| c | c c c c c }
\hline
\diagbox{p}{q} & \multicolumn{1}{c|}{$1$} & \multicolumn{1}{m{2.5em}|}{\centering$2$} & \multicolumn{1}{m{2.5em}|}{\centering$3$} & \multicolumn{1}{m{2.5em}|}{\centering$4$} &\multicolumn{1}{m{2.5em}}{\centering$\ge5$} \\ \hline
1 & \multicolumn{1}{c|}{$n$} & \multicolumn{4}{c}{$\mathbf{1}$}                                                        \\ \hline
2 & \multicolumn{1}{c|}{$\mathbf{\bigl\lceil\frac{n+m+1}{2}\bigr\rceil}$}  & \multicolumn{4}{c}{$\bigl\lceil \frac{n+m}{2}\bigr\rceil$}                                                         \\ \hline
3 & \multicolumn{1}{c|}{}  & \multicolumn{1}{c|}{}  & \multicolumn{1}{c|}{}  & \multicolumn{2}{c}{$\sigma_{(3,\infty)}$}    \\ \hline
4 & \multicolumn{5}{c}{$\mathbf{2n + 2m + \bigl\lfloor \frac{(m-2)(n-2)}{2}\bigr\rfloor}$}                                                                         \\ \hline
$\ge 5$ & \multicolumn{5}{c}{$nm$}                      
\end{tabular}
\caption{$(p,q)$-spreading numbers of grids $P_m\Box P_n$. Values in bold are results in this paper.}
\label{Tabela1}
\end{center}
\end{table}

%%%%%%%%%%%%%%%%%%%%%%
\section{Concluding remarks}
\label{sec:conclude}
%%%%%%%%%%%%%%%%

It follows from Remark~\ref{r:basic} that $\sigma_{(3,4)}(P_m\cp P_n)=\sigma_{(3,\infty)}(P_m\cp P_n)$ for any $m,n\ge 3$, since $\Delta(P_m\cp P_n)=4$. Unfortunately, the exact values of the bootstrap percolation numbers $m(P_m\cp P_n,3)$ ($=\sigma_{(3,\infty)}(P_m\cp P_n)$) 
are not known in general; see~\cite{HH-2023+,romer} for some known results and bounds, and~\cite{dnr-2022} for a related study. 
However, we believe that  $\sigma_{(3,3)}(P_m\cp P_n)= \sigma_{(3,4)}(P_m\cp P_n)=\sigma_{(3,\infty)}(P_m\cp P_n)$, and pose this as a conjecture.

\begin{conj}
\label{con:grid}
If $m,n\ge 3$, then $\sigma_{(3,3)}(P_m\cp P_n)= \sigma_{(3,4)}(P_m\cp P_n)$.
\end{conj}

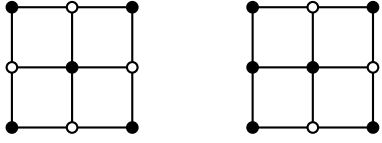
\begin{figure}[ht!]
\begin{center}
\begin{tikzpicture}[scale=0.8,style=thick,x=1cm,y=1cm]
\def\vr{2.5pt} % \vr = vertex radius;
% define vertices

\path (0,1) coordinate (a1);
\path (0,2) coordinate (b1);
\path (0,3) coordinate (c1);

\path (1,1) coordinate (a2);
\path (1,2) coordinate (b2);
\path (1,3) coordinate (c2);

\path (2,1) coordinate (a3);
\path (2,2) coordinate (b3);
\path (2,3) coordinate (c3);

\path (4,1) coordinate (A1);
\path (4,2) coordinate (B1);
\path (4,3) coordinate (C1);

\path (5,1) coordinate (A2);
\path (5,2) coordinate (B2);
\path (5,3) coordinate (C2);

\path (6,1) coordinate (A3);
\path (6,2) coordinate (B3);
\path (6,3) coordinate (C3);

%  edges
\draw (a1)--(b1)--(c1);
\draw (a2)--(b2)--(c2);
\draw (a3)--(b3)--(c3);

\draw (a1)--(a2)--(a3);
\draw (b1)--(b2)--(b3);
\draw (c1)--(c2)--(c3);

\draw (A1)--(B1)--(C1);
\draw (A2)--(B2)--(C2);
\draw (A3)--(B3)--(C3);

\draw (A1)--(A2)--(A3);
\draw (B1)--(B2)--(B3);
\draw (C1)--(C2)--(C3);

% vertices
\draw (a1) [fill=black] circle (\vr);
\draw (b1) [fill=white] circle (\vr);
\draw (c1) [fill=black] circle (\vr);

\draw (a2) [fill=white] circle (\vr);
\draw (b2) [fill=black] circle (\vr);
\draw (c2) [fill=white] circle (\vr);

\draw (a3) [fill=black] circle (\vr);
\draw (b3) [fill=white] circle (\vr);
\draw (c3) [fill=black] circle (\vr);

\draw (A1) [fill=black] circle (\vr);
\draw (B1) [fill=black] circle (\vr);
\draw (C1) [fill=black] circle (\vr);

\draw (A2) [fill=white] circle (\vr);
\draw (B2) [fill=black] circle (\vr);
\draw (C2) [fill=white] circle (\vr);

\draw (A3) [fill=black] circle (\vr);
\draw (B3) [fill=white] circle (\vr);
\draw (C3) [fill=black] circle (\vr);

\end{tikzpicture}
\end{center}

\caption{Minimum $(3,3)$- and $(3,1)$-spreading sets of $P_3 \cp P_3$.}
\label{fig:(3,1) ni enako (3,3)}
\end{figure}

One might ask, if the conjecture is proven to be true, whether one can extend the equality even further. While we did not find an answer to the question whether $\sigma_{(3,2)}(G) = \sigma_{(3,3)}(G)$ is true for all grids $G$, we can show that at least $\sigma_{(3,1)}(G)$ is not always equal to $\sigma_{(3,3)}(G)$ if $G$ is a grid. See the example of $G=P_3 \cp P_3$ in Figure~\ref{fig:(3,1) ni enako (3,3)}, from which it is clear that $\sigma_{(3,1)}(G) > \sigma_{(3,2)}(G) = \sigma_{(3,3)}(G)= \sigma_{(3,4)}(G)$, therefore in general $\sigma_{(3,1)}(P_m \cp P_n) \neq \sigma_{(3,3)}(P_m \cp P_n)$.

Even if one can prove Conjecture~\ref{con:grid}, this does not necessarily mean that one can determine $\sigma_{(3,3)}(P_m\cp P_n)$ or $\sigma_{(3,4)}(P_m\cp P_n)$. Based on some earlier partial results, this could be a hard problem, which we pose next. 

\begin{prob}
\label{prob-grids}
Determine $\sigma_{(3,q)}(P_m\Box P_n)$, where $q\le 4$, for all positive integers $m$ and $n$.  
\end{prob}

Obtaining upper and lower bounds on $\sigma_{(3,q)}(P_m\Box P_n)$ would be a natural way to approach Problem~\ref{prob-grids}. Alternatively, one could aim to obtain exact values while fixing one of the integers $m$ or $n$; we again refer to~\cite{HH-2023+,romer} for known values. 

Hypercubes, alias $n$-cubes, are one of the most important graph classes with numerous applications; one can define the $n$-cube, $Q_n$, as $K_2\cp\cdots\cp K_2$, where there are $n$ factors.  In particular, it is known that $\sigma_{(1,1)}(Q_n)=2^{n-1}$~\cite{AIM}. In addition, $\sigma_{(2,\infty)}(Q_n)=\lceil\frac {n}{2}\rceil+1$, as proved in~\cite{bal-2006}, and $\sigma_{(3,\infty)}(Q_n)=\lceil\frac {n(n+3)}{6}\rceil+1$, as proved in~\cite{MN}. A natural problem is to determine the $(p,q)$-spreading numbers of $n$-cubes for other values of $p$ and $q$. 

\begin{prob}
\label{prob1}
Determine $\sigma_{(p,q)}(Q_n)$, where $Q_n$ is the $n$-cube of dimension $n$ for all $p\in \mathbb{N},q\in \mathbb{N}\cup\{\infty\}$. 
\end{prob}

The most natural generalizations of hypercubes are Hamming graphs, which are the Cartesian products of several complete graphs. Problem~\ref{prob1} could thus be extended to determining $\sigma_{(p,q)}(K_{r_1} \Box\cdots\Box K_{r_n})$, where $K_{r_i}$ are complete graphs. This problem was considered for $P_3$-convexity in~\cite{bv-2019+}, where it was proved that $\sigma_{(2,\infty)}(K_{r_1} \Box\cdots\Box K_{r_n})=\lceil n/2\rceil +1$.

\section*{Acknowledgement}
B.B. was supported by the Slovenian Research Agency (ARRS) under the grants P1-0297, J1-2452, J1-3002, and J1-4008. T.D. was supported by the Slovenian Research Agency (ARRS) under the grants P1-0297 and J1-2452. J.H. was supported by the Slovenian Research Agency (ARRS) under the grants P1-0297.

\end{document}